\documentclass[12pt]{article}
\usepackage[top=2.5cm,bottom=2.5cm,left=2.9cm,right=2.9cm]{geometry}
\usepackage{amssymb}
\usepackage{amsmath,amsthm}
\usepackage{graphicx}
\usepackage{enumerate}
\usepackage{tikz}
\usepackage{mathrsfs}
\usepackage{verbatim}
\usepackage{upgreek}
 \usepackage{mathptmx}

\usepackage{upgreek}
\usepackage{mathptmx}
\usepackage{eucal}

\usepackage{colortbl}

\newcommand{\epn}{\operatorname{epn}}

\newtheorem{remark}{Remark}[section]

\newtheorem{theorem}[remark]{Theorem}

\title{A note on double domination in graphs}

\author{Abel Cabrera Mart\'inez, Juan A. Rodr\'{\i}guez-Vel\'{a}zquez\\
{\small Universitat Rovira i Virgili, } \\ 
{\small Departament d'Enginyeria Inform\`atica i Matem\`atiques,} \\ {\small Av. Pa\"{\i}sos
Catalans 26, 43007 Tarragona, Spain.} \\{\small
  abel.cabrera\@@urv.cat, juanalberto.rodriguez\@@urv.cat}
}

\date{ }
\begin{document}
\maketitle

\begin{abstract}
Recently, Haynes, Hedetniemi and  Henning published the book 
\emph{Topics in Domination in Graphs}, which comprises 16 contributions that present advanced topics in graph domination, featuring open problems, modern techniques, and recent results. One of these contributions is the chapter \emph{Multiple Domination}, by Hansberg and Volkmann, where they put into context all relevant research results on multiple domination that have been found up to 2020. In this note,  we show how to improve some results on double domination that are included in the book.
\end{abstract}


\section{Introduction}

Given a simple graph $G$,  the \emph{open neighbourhood} of a vertex $v\in V(G)$ is defined to be $N(v)=\{u \in V(G):\; u \text{ is adjacent to } v\}$, and the \emph{closed neighbourhood} of $v$ is $N[v]=N(v)\cup\{v\}$. 

Harary and Haynes \cite{MR1401362,Harary2000}  introduced, in two different papers published in
1996 and 2000, the concept of double domination and, more generally, the concept of $k$-tuple
domination. Given a positive integer $k$, a subset $D\subseteq  V (G)$ is said to be a $k$-\emph{tuple dominating set} if 
$|D\cap N[v]|\ge k$
for every  vertex $v\in V (G)$. 
 Of course, this definition requires that the graph in question has minimum degree at
least $k-1$.  

The $k$-\emph{tuple domination number}  $\gamma_{\times k}(G)$  is defined to be $$\gamma_{\times k}(G)=\min\{|D|:\, D \text{ is a  } k \text{-tuple dominating set}\}.$$ 
  The case $k=2$ corresponds to double domination,  and $\gamma_{\times 2}(G)$ is known as the \emph{double domination number}, while for $k=1$, the $1$-tuple dominating sets are the classical \emph{dominating sets}, and hence the \emph{domination number} is 
$\gamma(G) =\gamma_{\times 1}(G)$.

For a comprehensive survey on $k$-tuple domination in graphs, we cite the book \emph{Topics in Domination
in Graphs}, published in 2020, which was edited by Haynes, Hedetniemi and  Henning \cite{Book-TopicsDom-2020}.  
In particular, there is a  chapter, \emph{Multiple Domination}, by Hansberg and  Volkmann, where they
put into context all relevant research results on multiple domination concerning $k$-domination, $k$-tuple domination, and total $k$-domination that have been found up to 2020. In this note, we show how to improve some of the results on double domination that are included in the book.

\section{Upper bounds on the double domination number}

We begin with the following definitions.
We define a  $\gamma(G)$-set as a dominating set $D$
with $|D|=\gamma(G)$. The same agreement will be assumed for optimal parameters associated to other characteristic sets defined in the paper.

Recall that a set of vertices of a graph, no two of which are adjacent, is called an \emph{independent set}. The \emph{independence number} $\alpha(G)$ is defined to be 
$$\alpha(G)=\max\{|S|:\, S \text{ is an independent  set}\}.$$ 
A set $S\subseteq V(G)$ which is both dominating and independent is called an \emph{independent dominating set}. Moreover, the \emph{independent domination number} $i(G)$ is
$$i(G)=\min\{|S|:\, S \text{ is an independent dominating set}\}.$$

In 2007, Blidia, Chellali  and Favaron \cite{BCF2007} proved the following relationship between the double domination number, the independent domination number and the independence number of a graph.

\begin{theorem} {\rm \cite{BCF2007}} \label{teo-BCF2007-alpha-i}
For any graph  $G$  with no isolated vertex,
$$\gamma_{\times 2}(G)\leq \alpha(G)+i(G).$$
\end{theorem}

As shown in \cite{BCF2007}, the bound above is tight. Even so, the next result shows that, in a certain sense, Theorem \ref{teo-BCF2007-alpha-i} has room for improvement. 
 In fact,
 this can be seen as a corollary of Theorem \ref{teo-alpha-gamma}, as $\gamma(G)\le i(G)$.

\begin{theorem}\label{teo-alpha-gamma}
For any graph  $G$  with no isolated vertex,
$$\gamma_{\times 2}(G)\leq \alpha(G)+\gamma(G).$$
\end{theorem}

\begin{proof}
Let $S$ be an $\alpha(G)$-set and $D$ a $\gamma(G)$-set. 
Given a  vertex $u\in S\cup D$, the external private neighbourhood of $u$ with respect to $S\cup D$ is defined to be $$\epn(u,S\cup D)=\{v\in V(G)\setminus (S\cup D): \, N(v)\cap (S\cup D)=\{u\}\}.$$  
Note that if $u\in S$, then  $\epn(u,S\cup D)\subseteq \epn(u,S)$.

Now, we define  $W'\subseteq V(G)$ as a set of minimum cardinality among all supersets  $W$ of $S\cup D$ such that   the following properties are satisfied.  
\begin{itemize}
\item[(a)] If $x\in S\cap D$ and $\epn(x, S\cup D)\ne  \varnothing $, then $W\cap \epn(x, S\cup D)\neq \varnothing $.
\item[(b)] If $x\in S\cap D$ and $\epn(x, S\cup D)= \varnothing $, then  $ W\cap N(x)\ne \varnothing $.
\end{itemize}
By definition,  
$W'$ is a dominating set of $G$
whose induced subgraph has no isolated vertex
and $|W'|\leq |S|+|D|$.   To conclude that $W'$ is a double dominating set of $G$, we only need to prove that $|N(v)\cap W'|\ge 2$ for every    $v\in V(G)\setminus W'$. Suppose to the contrary that there exists $v\in V(G)\setminus W'$  such that $v\in \epn(u,W')$ for some vertex $u\in W'$. Now, since $S$ and $D$ are dominating sets of $G$, we have that $u\in S\cap D$. By (a), there exists $v'\in W'\cap \epn(u, S\cup D)$.  Since $v,v'\in \epn(u,S)$ and $\epn(u,S)$ forms a clique, vertex  $v$ has at least two neighbours  $u,v'\in W'$,  which is a contradiction. Therefore,  $W'$ is a double dominating set of $G$ and so,  $\gamma_{\times 2}(G)\leq |W'|\leq |S|+|D|= \alpha(G)+\gamma(G)$.
\end{proof}

Let $\mathcal{H}$ be the family of graphs $G_{t,r}$ defined as follows. For any pair of integers $t, r\in \mathbb{Z}$, with $t\ge 2$ and $1\le  r\le t-1$, the graph $G_{t,r}$ is obtained from two different copies of a star  $T_1\cong T_2\cong K_{1,t}$ such that   $V(G_{t,r})=V(T_1)\cup V(T_2)$,  $V(T_1)=\{u, v_1,\dots, v_t\}$ and $V(T_2)=\{u', v_1',\dots, v_t'\}$,  $u$ and $u'$  are the support vertices of the stars, and $E(G_{t,r})=E(T_1)\cup E(T_2)\cup \{uu',v_1v_1',\dots, v_rv_r'\}.$ 	
Figure \ref{figure} shows a graph of this family. 
Observe that $\gamma_{\times 2}(G_{t,r})=2t-r+2$, $\gamma(G_{t,r})=2$, $i(G_{t,r})=t+1$ and $\alpha(G_{t,r})=2t-r$  for every $G_{t,r}\in \mathcal{H}$. Therefore, for these graphs Theorem \ref{teo-alpha-gamma} gives the exact value $\gamma_{\times 2}(G_{t,r})=\alpha(G_{t,r})+\gamma(G_{t,r})$, while Theorem \ref{teo-BCF2007-alpha-i} gives $\gamma_{\times 2}(G_{t,r})\le \alpha(G_{t,r})+\gamma(G_{t,r})+t-1=\alpha(G_{t,r})+i(G_{t,r})$.

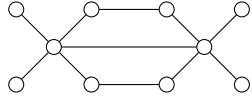
\begin{figure}[ht]
\centering
\begin{tikzpicture}[scale=.5, transform shape]
\node [draw, shape=circle] (c1) at  (0,0) {};
\node [draw, shape=circle] (h1) at  (-1,-1) {};
\node [draw, shape=circle] (h2) at  (1,1) {};
\node [draw, shape=circle] (h3) at  (-1,1) {};
\node [draw, shape=circle] (h4) at  (1,-1) {};
\node [draw, shape=circle] (c2) at  (4,0) {};
\node [draw, shape=circle] (hh1) at  (5,-1) {};
\node [draw, shape=circle] (hh2) at  (3,1) {};
\node [draw, shape=circle] (hh3) at  (5,1) {};
\node [draw, shape=circle] (hh4) at  (3,-1) {};
\draw (h4)--(c1)--(c2);
\draw (c1)--(h1);
\draw (c1)--(h2);
\draw (c1)--(h3);
\draw (c2)--(hh1);
\draw (c2)--(hh2);
\draw (h4)--(hh4)--(c2)--(hh3);
\draw (h2)--(hh2);
\end{tikzpicture}
\caption{A graph $G_{4,2}\in \mathcal{H}$.}\label{figure}
\end{figure}

A set $S$ of vertices of $G$ is a \emph{vertex cover} if every edge of $G$ is incident with at least one vertex in $S$. The \emph{vertex cover number} of $G$, denoted by $\beta(G)$, is defined as
$$\beta(G)=\min\{|S|:\, S \text{ is a vertex cover}\}.$$

\begin{theorem}\label{teo-consequence-Cabrera-DOIDS-2020}
Let $G$ be a graph with no isolated vertex of order $n\geq 3$. Let $\mathcal{L}(G)$ and $\mathcal{S}(G)$ be the set of leaves and support vertices of $G$, respectively. Then  $$\gamma_{\times 2}(G)\leq \beta(G) +\gamma(G)+|\mathcal{L}(G)|-|\mathcal{S}(G)|.$$
\end{theorem}

\begin{proof}
Let $S$ be a $\beta(G)$-set and $D$ a $\gamma(G)$-set such that $\mathcal{S}(G)\subseteq S\cap D$. Let us define  $W'\subseteq V(G)$ as a set of minimum cardinality among all sets $W$ satisfying the following properties.

\begin{itemize}
\item $S\cup D\cup \mathcal{L}(G)\subseteq W$.
\item If $v\in (S\cap D)\setminus \mathcal{S}(G)$, then $N(v)\cap W\ne \varnothing$. 
\end{itemize} 
Since $S\cup \mathcal{L}(G)\subseteq W'$, we conclude that  every vertex outside $W'$ is adjacent to at least two vertices in $W'$.   Now, let $x\in W'$. If $x\in S\cap D$, then $N(x)\cap W'\neq \varnothing$ by definition. If $x\in (S\cup D)\setminus (S\cap D)$, then $N(x)\cap W'\neq \varnothing$, as $S$ and $D$ are dominating sets of $G$.  Therefore, $W'$ is a double dominating set of $G$, and since  $\mathcal{S}(G)\subseteq S\cap D$ we obtain that $\gamma_{\times 2}(G)\leq |W'| \leq |S|+|D|+|\mathcal{L}(G)|-|\mathcal{S}(G)|=\beta(G)+\gamma(G)+|\mathcal{L}(G)|-|\mathcal{S}(G)|$, which completes the proof. 
\end{proof}

The bound above is tight. For instance, it is achieved by any graph $G_{t,r}\in \mathcal{H}$, where $\mathcal{H}$ is the family of graphs defined after the proof of Theorem \ref{teo-alpha-gamma}. In this case,
$\gamma_{\times 2}(G_{t,r})=2t-r+2$, $\gamma(G_{t,r})=|\mathcal{S}(G_{t,r})|=2$, $|\mathcal{L}(G_{t,r})|=2(t-r)$ and $\beta(G_{t,r})=r+2$  for every $G_{t,r}\in \mathcal{H}$.

In 2000, Harary and Haynes \cite{Harary2000} gave the following upper bound on the double domination number of a graph in terms of the order and the domination number.

\begin{theorem}{\rm \cite{Harary2000}}\label{teo-HH2000}
For any graph  $G$ of order $n$ and minimum degree $\delta(G)\geq 2$, 
$$\gamma_{\times 2}(G)\leq \left\{ \begin{array}{ll}
             \lfloor \frac{n}{2} \rfloor + \gamma(G), & \text{if } \, n=3,5\\[5pt]
             \lfloor \frac{n}{2} \rfloor + \gamma(G)-1, & \mbox{otherwise.}
                                  \end{array}\right.$$
                                  Furthermore, the bound is tight.
\end{theorem}

We proceed to show that Theorems \ref{teo-alpha-gamma} and  \ref{teo-consequence-Cabrera-DOIDS-2020} allow us to improve the bound given in Theorem \ref{teo-HH2000}. To this end, we need to state the following well-known result, due to Gallai, which states the relationship between the independence number and the vertex cover number of a graph.

\begin{theorem}{\em\cite{Gallai1959}}{\rm (Gallai's theorem)}\label{th_gallai}
For any graph  $G$ of order $n$,
$$\alpha(G)+\beta(G) = n.$$
\end{theorem}

By combining Theorems \ref{teo-alpha-gamma}, \ref{teo-consequence-Cabrera-DOIDS-2020} and \ref{th_gallai}, we have the next result, which clearly improves the bound given in Theorem \ref{teo-HH2000} for graphs of minimum
degree at least two and order $n\ge 6$, whenever $\alpha(G)\notin \{    \frac{n\pm 2}{2}   ,   \frac{n}{2} \}$ for $n$ even, and $\alpha(G)\notin \{   \frac{n\pm 1}{2}  ,  \frac{n \pm 3}{2} \}$ for $n$ odd.

\begin{theorem}\label{teo-improve-HH2000}
For any graph  $G$ of order $n$ and minimum degree $\delta(G)\geq 2$, 
$$\gamma_{\times 2}(G)\leq \min\{\alpha(G),n-\alpha(G)\}+\gamma(G).$$
\end{theorem}

The bound above is tight. For instance, it is achieved by any complete graph, as $\gamma_{\times 2}(K_n)=\alpha(K_n)+\gamma(K_n)=2$, while Theorem \ref{teo-HH2000} gives   $\gamma_{\times 2}(K_n)\le   \left\lfloor \frac{n}{2} \right\rfloor $ for $n> 5$ and $\gamma_{\times 2}(K_5)\le 3$.
Now, in order to show another example where 
the difference between the bounds given by Theorems \ref{teo-improve-HH2000} and \ref{teo-HH2000} can be as large as desired,
we consider the family  $\mathcal{H}'$ of graphs defined as follows. For every integer $r\ge 2$ the graph $G_r\in \mathcal{H}'$ has order $n=3(r+1)$, as $V(G_r)=\{a_1, a_2, a_3, v_1, \dots, v_{3r}\}$,   and the edges of $G_r$ are of the form $\{v_j,a_i\}$ and  $\{v_j,a_{i+1}\}$ whenever $j\equiv i \pmod 3$. 
  In this case, 
$\gamma_{\times 2}(G_r)=5,$ Theorem  \ref{teo-improve-HH2000} gives  $\gamma_{\times 2}(G_r)\le 3(r+1)-\alpha(G_r)+\gamma(G_r)=6,$ while Theorem \ref{teo-HH2000} gives  $\gamma_{\times 2}(G_r)\le   \left\lfloor \frac{3(r+1)}{2} \right\rfloor +2$. Figure \ref{Fig2}  shows the case $r=2$. 

\begin{figure}[ht]
\centering
\begin{tikzpicture}[scale=.5, transform shape]
\node [draw, shape=circle] (a1) at  (2,0) {};
\node [draw, shape=circle] (a2) at  (3.5,0) {};
\node [draw, shape=circle] (a3) at  (5,0) {};
\node [draw, shape=circle] (v1) at  (1,2) {};
\node [draw, shape=circle] (v2) at  (2,2) {};
\node [draw, shape=circle] (v3) at  (3,2) {};
\node [draw, shape=circle] (v4) at  (4,2) {};
\node [draw, shape=circle] (v5) at  (5,2) {};
\node [draw, shape=circle] (v6) at  (6,2) {};
\draw (a1)--(v1)--(a2)--(v2)--(a3)--(v3)--(a1)--(v4)--(a2)--(v5)--(a3)--(v6)--(a1);
\node [below] at (2,-0.2) {\LARGE $a_1$};
\node [below] at (3.5,-0.2) {\LARGE $a_2$};
\node [below] at (5,-0.2) {\LARGE $a_3$};
\node [above] at (1,2.2) {\LARGE $v_1$};
\node [above] at (2,2.2) {\LARGE $v_2$};
\node [above] at (3,2.2) {\LARGE $v_3$};
\node [above] at (4,2.2) {\LARGE $v_4$};
\node [above] at (5,2.2) {\LARGE $v_5$};
\node [above] at (6,2.2) {\LARGE $v_6$};
\end{tikzpicture}
\caption{The graph $G_2\in \mathcal{H}'$.}\label{Fig2}
\end{figure}
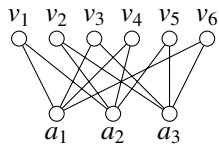

In 1985, Fink and
Jacobson \cite{MR812671,MR812672}  extended  the idea of domination in graphs  to the more general notion of $k$-domination. A subset $D\subseteq V(G)$ is said to be a $k$-dominating set if $|D\cap N(v)|\ge k$ for every vertex $v\in  V (G) \setminus  D$.  The $k$-domination number $\gamma_k(G)$ is defined to be 
$$\gamma_k(G)=\min\{|D|:\, D \text{ is a } k\text{-dominating set}\}.$$
The case $k=1$ corresponds to standard domination. 

Bonomo et al. \cite{BBGMS2018} obtained the following relationship between the double domination number and the $2$-domination number.

\begin{theorem}{\rm \cite{BBGMS2018}}\label{Bonono}
For any graph  $G$ with no isolated vertex, 
$$\gamma_{\times 2}(G)\leq 2\gamma_2(G)-1.$$
\end{theorem}
The inequality of this theorem was proved previously for trees
in \cite{MR2418091}.
In fact, the result was stated as $\gamma_{\times 2}(T)\leq 2\gamma_2(T)-2$ for every tree $T$ of order at least four.

Since $\gamma(G)\le\gamma_2(G)$, the next result improves the bound  given in Theorem \ref{Bonono} whenever $\gamma(G)\le \gamma_2(G)-2$.

\begin{theorem}\label{2+gamma}
For any graph  $G$ with no isolated vertex,
$$\gamma_{\times 2}(G)\leq \gamma_2(G)+\gamma(G).$$
\end{theorem}

\begin{proof}
Let $S$ be a $\gamma_2(G)$-set and $D$ a $\gamma(G)$-set. Now, we define  $W'\subseteq V(G)$ as a set of minimum cardinality among all supersets  $W$ of $S\cup D$ such that   $N(x)\cap W\ne \varnothing$ for every $x\in S\cap D$.

By definition, $W'$ is a double dominating set of $G$ and $|W'|\leq |S|+|D|$. Therefore,  $\gamma_{\times 2}(G)\leq |W'|\leq |S|+|D|= \gamma_2(G)+\gamma(G)$, which completes the proof.
\end{proof}

The bound above is tight. For instance, it is achieved by the graph $G$ shown in Figure \ref{Fig3}, where $\gamma(G)=2$, $\gamma_2(G)=4$ and $\gamma_{\times 2}(G)=6$. For a simple example where 
the difference between the bounds given by Theorems \ref{2+gamma} and \ref{Bonono} can be as large as desired, we can take $G\cong K_{1,n-1}$ with $n\ge 3$. In this case, Theorem
 \ref{2+gamma} gives the exact value  $\gamma_{\times 2}(K_{1,n-1})=n$, while Theorem  \ref{Bonono} gives $\gamma_{\times 2}(K_{1,n-1})\le 2n-3$.

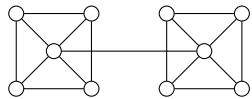
\begin{figure}[ht]
\centering
\begin{tikzpicture}[scale=.5, transform shape]
\node [draw, shape=circle] (c1) at  (0,0) {};
\node [draw, shape=circle] (h1) at  (-1,-1) {};
\node [draw, shape=circle] (h2) at  (1,1) {};
\node [draw, shape=circle] (h3) at  (-1,1) {};
\node [draw, shape=circle] (h4) at  (1,-1) {};
\node [draw, shape=circle] (c2) at  (4,0) {};
\node [draw, shape=circle] (hh1) at  (5,-1) {};
\node [draw, shape=circle] (hh2) at  (3,1) {};
\node [draw, shape=circle] (hh3) at  (5,1) {};
\node [draw, shape=circle] (hh4) at  (3,-1) {};
\draw (c2)--(c1);
\draw (h4)--(c1);
\draw (c1)--(h1);
\draw (c1)--(h2);
\draw (c1)--(h3);
\draw (c2)--(hh1);
\draw (c2)--(hh2);
\draw (hh4)--(c2)--(hh3);
\draw (h1)--(h3)--(h2)--(h4)--(h1);
\draw (hh1)--(hh3)--(hh2)--(hh4)--(hh1);
\end{tikzpicture}
\caption{A graph $G$ with $\gamma_{\times 2}(G)= \gamma_2(G)+\gamma(G)=6<7=2\gamma_2(G)-1$.}\label{Fig3}
\end{figure}

A \emph{total dominating set}, in a graph $G$ with no isolated vertices, 
is a set $S$ of vertices of $G$ such that every vertex of $G$ is adjacent to at least one vertex
in $S$. 
The \emph{total domination number} $\gamma_t(G)$ is defined to be 
$$\gamma_t(G)=\min\{|S|:\, S \text{ is a total dominating set}\}.$$  
The theory of total domination  has been extensively studied. For instance, we cite the book  \cite{Henning2013} by Henning and Yeo.

A graph is \emph{claw-free} if and only if it does not contain the complete bipartite graph $K_{1,3}$ as an induced subgraph. 

\begin{theorem}{\rm \cite{BCF2007}}\label{teo-BCF2007-gamma-t}
For any claw-free graph $G$ with no isolated vertex, 
$$\gamma_{\times 2}(G)\leq \min\{2\gamma_t(G),3\gamma(G)\}.$$
\end{theorem}

As shown in \cite{BCF2007}, the bound above is tight. Even so, the next result shows that  Theorem \ref{teo-BCF2007-gamma-t} has room for improvement, as $\gamma(G)\le \gamma_t(G)$ and $\gamma_t(G)\le 2 \gamma(G)$.

\begin{theorem}\label{teo-improve-claw-free}
For any claw-free graph $G$ with no isolated vertex,
$$\gamma_{\times 2}(G)\leq \gamma_t(G)+\gamma(G).$$
\end{theorem}

\begin{proof}
Let $S$ be a $\gamma_t(G)$-set and $D$ a $\gamma(G)$-set. As above, given a  vertex $v\in S\cup D$, the external private neighbourhood of $v$ with respect to $S\cup D$  will be denoted by $\epn(v,S\cup D)$.

Now, we define  $W'\subseteq V(G)$ as a set of minimum cardinality among all supersets  $W$ of $S\cup D$ such that  the following property is satisfied.
\begin{itemize}
\item[($\star$)] If $x\in S\cap D$ and $\epn(x, S\cup D)\ne  \varnothing $, then $W\cap \epn(x, S\cup D)\neq \varnothing $.
\end{itemize}

By definition,  $W'$ is a total dominating set of $G$ and $|W'|\leq |S|+|D|$. It remains to show that $|N(v)\cap W'|\geq 2$ for every  $v\in V(G)\setminus W'$. Suppose, to the contrary, that there exists $v\in V(G)\setminus W'$  which is an external private neighbour of $u\in W'$, i.e., $N(v)\cap W'=\{u\}$. Since both $S$ and $D$ are dominating sets,   $u\in S\cap D$. Now, by ($\star$), there exists  $v'\in  W'\cap ( \epn(u,S\cup D)\setminus \{v\})$ and, since $S$ is a total dominating set, there exists   $u'\in N(u)\cap S$.  Hence, if  $G$ is a claw-free graph, then $v$ and $v'$ have to be adjacent vertices, which is a contradiction. Thus, $W'$ is a double dominating set of $G$. Therefore, 
  $\gamma_{\times 2}(G)\leq |W'|\leq |S|+|D|= \gamma_t(G)+\gamma(G)$, which completes the proof.
\end{proof}

The bound above is tight. For instance, it is achieved by any graph of order $n+1$ obtained from a complete graph of order $n\ge 2$ by adding a pendant edge.

To show clear examples where 
the difference between the bounds given by Theorems \ref{teo-improve-claw-free} and \ref{teo-BCF2007-gamma-t} can be as large as desired, we can take the case of path and cycle graphs of order large enough.

\end{document}